\documentclass[12pt,a4paper]{amsart}
\usepackage{amsmath,amsthm,amssymb,amsfonts}

\usepackage{graphicx}
\usepackage{comment}
\usepackage{hyperref}
\usepackage{paralist}

\usepackage[utf8]{inputenc}

\usepackage[usenames, dvipsnames]{color}
 
\definecolor{myorange1}{rgb}{0.85, 0.38, 0.2}
\definecolor{mygreen1}{rgb}{0.1, 0.4, 0.1}

\newtheorem{theorem}{Theorem}[section]
\newtheorem{lemma}[theorem]{Lemma}

\newtheorem{proposition}[theorem]{Proposition}

\theoremstyle{definition}

\newtheorem{remark}[theorem]{Remark}
\numberwithin{equation}{section}

\newcommand{\N}{\mathbb{N}} 
\newcommand{\Z}{\mathbb{Z}} 
\newcommand{\Q}{\mathbb{Q}} 
\newcommand{\R}{\mathbb{R}} 
\newcommand{\C}{\mathbb{C}} 
\newcommand{\D}{\mathbb{D}} 
\newcommand{\T}{\mathbb{T}} 
\newcommand{\e}{\varepsilon}
\newcommand{\CC}{\mathbb C}
\newcommand{\NN}{\mathbb N}

\newcommand{\HC}{\mathcal{H}^\infty(\mathbb{C}_+)}
\newcommand{\HCK}{\mathcal{H}^\infty(\mathbb{C}_+^k)}

\newcommand{\HCdos}{\mathcal{H}^\infty(\mathbb{C}_+^2)}

\DeclareMathOperator{\re}{Re}
\DeclareMathOperator{\im}{Im}
\DeclareMathOperator{\Arg}{Arg}

\def\N{{\mathbb N}}

\def\R{{\mathbb R}}

\def\C{{\mathbb C}}

\def\ignora#1{}
\def\n3#1{\left\vert  \! \left\vert \! \left\vert \, #1 \, \right\vert \!
  \right\vert \! \right\vert }

\title
{Composition operators on spaces of double Dirichlet series}

\author[F. Bayart]{Fr\'{e}d\'{e}ric Bayart}
\address{ Laboratoire de Math\'ematiques Blaise Pascal, Universit\'e Clermont Auvergne,
Campus des C\'ezeaux 3, place Vasarely TSA 60026 CS 60026 63178 Aubi\`ere cedex, France} \email{Frederic.Bayart@uca.fr}

\author[J. Castillo-Medina]{Jaime Castillo-Medina}
\address{Departamento de An\'{a}lisis Matem\'{a}tico,
Universidad de Valencia, Doctor Moliner 50, 46100 Burjasot
(Valencia), Spain} \email{jaime.castillo@uv.es}

\author[D. Garc\'{\i}a]{Domingo Garc\'{\i}a}
\address{Departamento de An\'{a}lisis Matem\'{a}tico,
Universidad de Valencia, Doctor Moliner 50, 46100 Burjasot
(Valencia), Spain} \email{domingo.garcia@uv.es}

\author[M. Maestre]{Manuel Maestre}
\address{Departamento de An\'{a}lisis Matem\'{a}tico,
Universidad de Valencia, Doctor Moliner 50, 46100 Burjasot
(Valencia), Spain} \email{manuel.maestre@uv.es}

\author[P. Sevilla-Peris]{Pablo Sevilla-Peris}
\address{Instituto Universitario de Matemática Pura y Aplicada,
Universitat Polit\`{e}cnica de Val\`encia, cmno Vera s/n, 46022,
Val\`encia, Spain} \email{psevilla@mat.upv.es }

\thanks{The first author was partially supported by the grant ANR-17-CE40-0021 of the French National Research Agency ANR (project Front). The last four authors were supported by MINECO and FEDER projects MTM2014-57838-C2-2-P and MTM2017-83262-C2-1-P. The second author was also supported by grant FPU14/04365 and MICINN. The third and fourth authors were also supported by  PROMETEO/2017/102}

\keywords{Double Dirichlet series, Composition operator, Superposition operator}

\begin{document}

\begin{abstract}
We study composition operators on spaces of double Dirichlet series, focusing our interest 
on the characterization of the composition operators of the space of bounded double Dirichlet series $\HCdos$.
We also show how the composition operators of this space of Dirichlet series are related to the composition operators of the corresponding spaces of holomorphic functions.
Finally, we give a characterization of the superposition operators in $\HC$ and in the spaces $\mathcal{H}^p$.
\end{abstract}

\maketitle

\section{Introduction}

The study of composition operators appears as a consistent topic of interest in the literature of Banach spaces of holomorphic functions. Generally, a composition operator is defined by a function $\phi$, called the \emph{symbol} of the composition operator $C_\phi$, so that $C_\phi(f) = f \circ \phi$. Composition operators of spaces of Dirichlet series were first studied in \cite{Gordon_Hedenmalm99}, where the authors focus on $\mathcal{H}^2$, the Hilbert space of Dirichlet series whose sequence of coefficients belongs to $\ell^2$. Before we continue, let us introduce some notation. We will denote by $\C_\sigma$, with $\sigma > 0$, the half-plane of complex numbers whose real part is strictly larger than $\sigma$, using $\C_+$ as a substitute notation for the case of $\C_0$ to highlight the relevance of this special case. We denote by $\mathcal{D}$ the set of Dirichlet series which converge in some half-plane, that is, the set of Dirichlet series which converge somewhere, hence in a half-plane (see \cite[Lemma~4.1.1]{Queffelec13}). The two main results of \cite{Gordon_Hedenmalm99} (see also \cite[Theorem~6.4.5]{Queffelec13}) characterize  composition operators acting on $\mathcal H^2$.

\begin{theorem} \cite[Theorem~A]{Gordon_Hedenmalm99} \label{Composition_Operators_on_H^2}
Let $\theta \in \mathbb{R}$. An analytic function $ \phi: \mathbb{C}_\theta \rightarrow \mathbb{C}_\frac{1}{2}$ generates a composition operator $C_{\phi} : \mathcal{H}^2 \rightarrow \mathcal{D}$ if and only if it is of the form $\phi (s) = c_0 s + \varphi(s)$ with $c_0 \in \mathbb{N}_0$ and $ \varphi \in \mathcal{D}$.
\end{theorem}

\begin{theorem}  \cite[Theorem~B]{Gordon_Hedenmalm99} \label{Composition_Operators_H^2}
An analytic function $ \phi: \mathbb{C}_{\frac{1}{2}} \rightarrow \mathbb{C}_{\frac{1}{2}}$ defines a bounded composition operator $C_{\phi} : \mathcal{H}^2 \rightarrow \mathcal{H}^2$ if and only if
\begin{itemize}
\item[(a)] it is of the form $\phi (s) = c_0 s + \varphi(s)$ with $c_0 \in \mathbb{N}_0$ and $ \varphi \in \mathcal{D}$.
\item[(b)] $\phi$ has an analytic extension to $\C_+$, also denoted by $\phi$, such that
\begin{itemize}
\item [(i)] $\phi(\C_+) \subset \C_+$ if $ 0 < c_0$, and
\item[(ii)] $\phi(\C_+) \subset \C_{\frac{1}{2}}$ if $ c_0=0$.
\end{itemize}
\end{itemize} 
\end{theorem}

Here it should be understood that $\varphi$ is a Dirichlet series which converges in some half-plane and has an analytic extension, also denoted by $\varphi$, defined in $\C_\theta$ for Theorem~\ref{Composition_Operators_on_H^2} and in $\C_\frac{1}{2}$ for Theorem~\ref{Composition_Operators_H^2}. The same applies for the results that appear throughout this work, where $\varphi$ will be a Dirichlet series convergent in a certain half-plane with an analytical extension to the domain of $\phi$. This work was extended to the whole scale of Hardy spaces of Dirichlet series in   \cite{Bayart02}.\\

In this paper, we are mostly interested in composition operators on $\HC$, the space of Dirichlet series that converge on some half-plane and that can be extended as a bounded holomorphic function
to $\mathbb{C}_{+}$ (see \cite[Chapter~6]{Queffelec13}) or, equivalently, of those Dirichlet series that converge on $\mathbb{C}_{+}$ to a bounded holomorphic function
(see \cite[Theorem~1.13]{DeGaMaSe18}). This space is endowed with the norm $\Vert D \Vert_{\infty} = \sup_{s \in \mathbb{C}_{+}} \vert D(s) \vert$.
It has been shown in \cite{Bayart02} that an analytic function $\phi:\CC_+\to\CC_+$ generates a composition operator $C_\phi:\HC\to\HC$ if and only if it writes $\phi(s)=c_0s+\varphi(s)$
with $c_0\in\NN_0$ and $\varphi\in\mathcal D$.

Our aim is to extend this result to double Dirichlet series and to characterize the composition operators on the space $\HCdos$, introduced in \cite{Nuestro} and whose definition we recall now. 
A double series $\sum_{m,n}a_{m,n}$ is \emph{regularly convergent} if it is convergent and if, for all $m_0$ and $n_0$, the series $\sum_n a_{m_0,n}$ and $\sum_m a_{m,n_0}$ are convergent.
Then, the space $\HCdos$ is defined as the space of double Dirichlet series $\sum_{m,n=1}^{+\infty}a_{m,n}m^{-t}n^{-s}$ that are regularly convergent on $\C_+^2$ to a bounded holomorphic function
(the norm being defined as the supremum on $\mathbb{C}_{+}^{2}$). With this we can state the main result of this paper.
It is an the analogue of the quoted result of \cite{Bayart02} in the double case. Note here that the hypothesis of analyticity of the symbol is also dropped.

\begin{theorem} \label{rural}
Consider a function $\phi: \C_+^2 \rightarrow \C_+^2$, $\phi=(\phi_1, \phi_2)$. Then $C_\phi : \HCdos \rightarrow \HCdos$ is a composition operator if and only if, for $j=1,2$ and $(s,t) \in \C_+^2$ we have
\begin{equation} \label{quinquennal}
\phi_j(s,t) = c_0^{(j)}s + d_0^{(j)} t + \varphi_j(s,t) \,,
\end{equation}
where 
\[
c_0^{(j)}, d_0^{(j)} \in \N_0
\] 
and 
\begin{gather*}
\varphi_j:\C_+^2 \rightarrow \C \\
\varphi_j(s,t) = \sum_{m,n=1}^\infty \frac{d_{m,n}^{(j)}}{m^s n^t}
\end{gather*}
is a double Dirichlet series that converges regularly and uniformly in $\C_\e^2$ for every $\e >0$.
\end{theorem}

The proof of this theorem will be divided into two parts. The proof that symbols satisfying \eqref{quinquennal} leads to composition operators on $\HCdos$ will follow the arguments of
the one-variable case, with extra difficulties since we are working with double Dirichlet series and we have to be very careful with the convergence of the series.
In particular, we will need the latest results about the range of the symbol of a composition operator of $\HC$, which are developed in \cite[Section~3]{Queffelec15}. The proof
of the converse part will require really new arguments, which were not necessary in the one-dimensional case (see in particular the proof of the forthcoming Theorem \ref{Comp_op_HCdos_necessity}.

A crucial fact in the theory of Dirichlet series is that Dirichlet series and formal power series (or functions) in infinitely many variables are related through the Bohr transform $\mathcal{B}$. 
Let us briefly recall how this identification is done. Each formal power series $\sum_{\alpha}a_\alpha z^\alpha$ (where $\alpha$ runs on the set of eventually zero sequences of non-negative integers) is mapped onto the Dirichlet series $\sum_{\alpha}a_\alpha(p_1^{\alpha_1}\cdots p_r^{\alpha_r})^{-s}$, where $(p_j)$ is the sequence of prime numbers. This defines an isometric isomorphism between $H^\infty(B_{c_0})$ (the space of bounded holomorphic functions on $B_{c_{0}}$, the open unit ball of $c_{0}$) and $\HC$ (this was proved in \cite{Hedenmalm_Lindqvist_Seip_97}, see also \cite[Theorem~3.8]{DeGaMaSe}) and between $H^p(\T^\infty)$ and $\mathcal{H}^p$ for every $1 \leq p \leq \infty$ \cite[Theorem~2]{Bayart02} (see also \cite[Proposition~6.53]{Queffelec13} or \cite[Chapter~11]{DeGaMaSe}). With this idea, it was shown in \cite{Bayart02} that, for $1 \leq p < \infty$  composition operators on the Hardy space $\mathcal{H}^p$ induce composition operators on $H^{p}(\mathbb{T}^{\infty})$. \\ 
This identification was extended to the double case in \cite[Theorem~3.5]{Nuestro}.
All this allows us to show (see Propositions~\ref{Comp_op_H_infty_B_c_0} and~\ref{Proposition_comp_ops_HCdos_H(B_c_0dos)}) that composition operators on $\HCK$ (for $k=1,2$) induce composition operators on $H^\infty(B_{c_0^{k}})$.

As a final note we will deal with superposition operators, since it is interesting to mention how there is a significant difference between the case of $\HC$ and the case of the spaces $\mathcal{H}^p$.

\section{Revisiting composition operators on $\HC$} \label{section_revisiting}

In this section, which is mostly expository, we revisit several arguments of the proof of \cite[Theorem~A]{Gordon_Hedenmalm99} and of the characterization of composition operators
on $\HC$. In particular, we clear some details in the proof that were implicit in the original papers and we improve the results by dropping the a priori assumption
of analyticity. All these results will be useful to study the double Dirichlet series case.

\begin{lemma} \label{Comp_Op_monomials}
Suppose that $\phi: \C_+ \rightarrow \C$ is a function such that $k^{-\phi} \in \HC$ for every $k \in \N$. Then $\phi$ is analytic in $\C_+$ and there exist $c_0 \in \N_0$, $\varphi\in\mathcal D$
such that $\varphi$ extends to $\CC_+$ and $\phi(s)=c_0 s + \varphi(s)$ for all $s \in \C_+$.
\end{lemma}

\begin{proof}
First let us prove that $\phi$ is analytic. By hypothesis, $D(s) = \frac{1}{2^{\phi(s)}} \in \HC$ with $|D(s)| > 0$ for every $s \in \C_+$.
Then there exists $L: \C_+ \rightarrow \C$ a holomorphic logarithm of $D$,
that is, $D(s)=e^{L(s)}$ for every $s \in \C_+$. As we have $e^{L(s)}= D(s) = e^{- (\log 2) \phi(s)}$,  there exists a function $k: \C_+ \rightarrow \Z$
such that $-(\log 2) \phi(s) = L(s) + 2k(s) \pi i$. Analogously for $\tilde{D}(s) = \frac{1}{3^{\phi(s)}}$, 
there exists $\tilde{L}: \C_+ \rightarrow \C$ a holomorphic logarithm of $\tilde{D}$ and a function $\tilde{k}: \C_+ \rightarrow \Z$ such that
$-(\log 3) \phi(s) = \tilde{L}(s) + 2\tilde{k}(s) \pi i$. Therefore, for every $s \in \C_+$,
$$ \frac{1}{-\log 2} (L(s) + 2k(s) \pi i) = \frac{1}{-\log 3} (\tilde{L}(s) + 2\tilde{k}(s) \pi i),$$
so, if $h : \C_+ \rightarrow \C$ is defined as
$$ h(s) = \frac{1}{2 \pi i} \left ( \frac{L(s)}{\log 2} - \frac{\tilde{L}(s)}{\log 3} \right ) = \frac{\tilde{k}(s)}{\log 3} - \frac{k(s)}{\log 2} \in \frac{1}{\log 2} \Z + \frac{1}{\log 3} \Z, $$
then $h(\C_+)$ is a countable set and by the open mapping property, $h$ is a constant function. Now suppose $\tilde{k}$ is not constant and take $s_1, s_2 \in \C_+$ such that $\tilde{k}(s_1) \neq \tilde{k}(s_2)$. Since $h(s_1) = h(s_2)$,
$$ \frac{\tilde{k}(s_1)}{\log 3} - \frac{k(s_1)}{\log 2} =  \frac{\tilde{k}(s_2)}{\log 3} - \frac{k(s_2)}{\log 2},$$
so
$$\frac{\log 3}{\log 2} = \frac{k(s_1) - k(s_2)}{\tilde{k}(s_1) -\tilde{k}(s_2)} \in \Q,$$
a contradiction. Therefore, $\tilde{k}$ is constant and $\phi(s) = \frac{1}{-\log 3} (\tilde{L}(s) + 2\tilde{k} \pi i)$ for all $s \in \C_+$, so $\phi$ is analytic in $\C_+$.

Now, note that in the proof of the necessity of \cite[Theorem~A]{Gordon_Hedenmalm99} the only hypothesis that is really used is that $k^{-\phi} \in \mathcal{D}$ for all $k \in \N$. 
Since the assumptions of the present lemma imply this, the proof now follows from the application of \cite[Theorem~A]{Gordon_Hedenmalm99}.
\end{proof}

\begin{remark}
Lemma~\ref{Comp_Op_monomials} can actually be stated with weaker hypothesis with essentially the same proof. It is enough to assume that $\phi : \C_\theta \rightarrow \C$ (for $\theta \geq 0$) is a function such that $k^{-\phi} \in \mathcal{D}$ for every $k \in \N$ to get that $\phi(s) = c_0 s + \varphi(s)$ for all $s \in \C_\theta$, with $c_0 \in \N$ and $\varphi \in \mathcal{D}$.
\end{remark}

The next result is a particular case of \cite[Proposition~4.2]{Gordon_Hedenmalm99}.

\begin{proposition} \label{comp_op_range_phi_C_+}
Suppose $\phi: \C_+ \rightarrow \C_+$ is a holomorphic mapping of the form $\phi(s)=c_0s + \varphi(s)$ for some $c_0 \in \N_0$ and such that $\varphi$ can be represented as a Dirichlet series that converges on $\C_\sigma$ for some $\sigma > 0$. Then for every $\e >0$ there exists some $\delta > 0$ such that $\phi(\C_{\e}) \subset \C_{\delta}$.
\end{proposition}

Our next theorem improves a result stated by the first author in \cite{Bayart02}, since we remove
the assumption of analyticity of the symbol. We give a detailed proof for the sake of completeness as it will be useful later on.

\begin{proposition} \label{Composition_Operators_on_H_inf(C_+)}
A symbol $ \phi: \mathbb{C}_+ \rightarrow \mathbb{C}_+$ generates a composition operator $C_{\phi} : \mathcal{H}^\infty (\mathbb{C}_+) \rightarrow \mathcal{H}^\infty (\mathbb{C}_+)$ 
if and only if it is analytic and, for all $s \in \C_+$, $\phi (s) = c_0 s + \varphi(s)$ with $c_0 \in \mathbb{N}_0$ and $\varphi \in \mathcal{D}$.
\end{proposition}

\begin{proof}
The necessity follows from Lemma~\ref{Comp_Op_monomials} and the fact that $\frac{1}{k^s} \in \HC$ for all $k \in \N$.
For the sufficiency, take $D \in \HC$ given by $D(s) = \sum_{k=1}^\infty \frac{a_k}{k^s}$ and $ \phi: \mathbb{C}_+ \rightarrow \mathbb{C}_+$ defined as $\phi (s) = c_0 s + \varphi(s)$ with $c_0 \in \mathbb{N}_0$ 
and $ \varphi \in \mathcal{D}$, and fix $\e > 0$. Consider $D_n(s) = \sum_{k=1}^n \frac{a_k}{k^s}$. By Theorem~\ref{Composition_Operators_on_H^2}, $D_n \circ \phi$ is a Dirichlet series that
converges in some half-plane. Moreover, by Proposition~\ref{comp_op_range_phi_C_+}, given $\e > 0$ there exists some $\delta > 0$ such that $\phi(\C_\e) \subset \C_\delta$ and therefore the sequence $\{D_n \circ \phi\}_n$ is uniformly bounded on $\C_\e$, say by $C$, as there it converges uniformly to $D \circ \phi$.
Write $(D_n \circ \phi)(s)= \sum_{k=1}^\infty \frac{b_k^{(n)}}{k^s}$. For a fixed $k$, the sequence $\{b_k^{(n)}\}_n$ is bounded by $C k ^\e$ and it is a Cauchy sequence since
\[
|b_k^{(n)}-b_k^{(m)}|
\leq k^{\e} \sup_{s \in \mathbb{C}_{\varepsilon}} \vert D_n\circ\phi (s) - D_m\circ\phi (s) \vert\,.
\] 
Therefore it converges to some $b_k \in \C$ with $|b_k| \leq C k^\e$. Defining $F(s)=\sum_{k=1}^\infty \frac{b_k}{k^s}$, since $|b_k| \leq C k^{\e}$, $F(s)$ converges absolutely in $\C_{1 + \e}$, and it is clear that in that half-plane the sequence $\{D_n \circ \phi\}_n$ converges to $F$. It is enough to note that $\Vert D \circ \phi \Vert_\infty \leq \Vert D \Vert_\infty < \infty$ to apply Bohr's Theorem (see e.g. \cite[Theorem~1.3]{DeGaMaSe} or \cite[Theorem~6.2.3]{Queffelec13}) and get that $D \circ \phi$ actually coincides with $F$ and that it is in $\HC$.
\end{proof}

%
%
%

Arguing as in \cite{Queffelec15}, we can improve Proposition \ref{Composition_Operators_on_H_inf(C_+)} and replace analyticity of $\varphi$ by its uniform convergence
on each half-plane. Indeed, under the previous assumptions, we know that $\varphi(\CC_+)\subset\CC_+$ (this is trivial if $c_0=0$ and follows from \cite[Proposition~4.2]{Gordon_Hedenmalm99}
if $c_0\neq 0$). Hence we may apply \cite[Theorem~3.1]{Queffelec15} which says the following:
\begin{quote}
 Suppose that $\varphi$ is analytic with no zeros in $\C_+$ and that the harmonic conjugate of $\log |\varphi|$ is bounded in $\C_+$. 
 If $\varphi$ can be represented as a convergent Dirichlet series in some half-plane $\C_{\sigma_0}$, then this Dirichlet series converges uniformly in $\C_\e$ for every $\e > 0$.
\end{quote}
We finally get an improved characterization of composition operators on $\HC$.

\begin{theorem} \label{Composition_Operators_on_H_inf(C_+)_FINAL}
A symbol $ \phi: \mathbb{C}_+ \rightarrow \mathbb{C}_+$ generates a composition operator $C_{\phi} : \mathcal{H}^\infty (\mathbb{C}_+) \rightarrow \mathcal{H}^\infty (\mathbb{C}_+)$ if and only if it is of the form $\phi (s) = c_0 s + \varphi(s)$ with $c_0 \in \mathbb{N}_0$ and $\varphi$ a Dirichlet series which converges uniformly in $\C_\e$ for all $\e > 0$.
\end{theorem}
Note that Theorem~\ref{rural} is a version of this result for double Dirichlet series.

\smallskip

We finish this section by pointing out that composition operators on $\HC$ induce composition operators on  $H^\infty(B_{c_0})$; the proof is omited as it follows with the same arguments as the corresponding result from the double case (see proof of  Proposition~\ref{Proposition_comp_ops_HCdos_H(B_c_0dos)}).

\begin{proposition} \label{Comp_op_H_infty_B_c_0}
Let $\phi:\C_+ \rightarrow \C_+$ be a symbol which defines a composition operator $C_\phi: \HC \rightarrow \HC$. Then there exists a map $\psi:B_{c_0}\to B_{c_0}$ such that $\mathcal B^{-1}\circ C_\phi\circ \mathcal B$ coincides with the composition operator $C_\psi: H^\infty(B_{c_0}) \rightarrow H^\infty(B_{c_0})$.
\end{proposition}

This relationship we have just stated only works one way, in the sense that there are composition operators on $H^{\infty} (B_{c_{0}})$ that do not induce a composition operator on $\HC$. The mapping $\psi :B_{c_0} \rightarrow B_{c_0}$ given by $\psi(z)=(z_1,0,\ldots)$ defines a composition operator $C_\psi$ on $H^\infty(B_{c_0})$. If there were $\phi: \C_+ \rightarrow \C_+$ such that $C_\psi= \mathcal{B}^{-1} \circ C_\phi \circ \mathcal{B}$, then $\mathcal{B}^{-1} (\frac{1}{3^{\phi(s)}}) = 0$ and hence $\frac{1}{3^{\phi(s)}} = 0$, which is a contradiction.

\section{Composition operators on $\HCdos$}

\subsection{The sufficient condition}

We can now start our work towards the characterization of composition operators on $\HCdos$ stated in Theorem~\ref{rural}. We do this in two steps. First we see that the symbols $\phi$ as in the statement of the theorem indeed define composition operators (see Theorem~\ref{Composition_Operators_on_H_inf(C_+^2)_sufficiency}). Once we have this we show that these are in fact the only symbols defining a composition operator on $\HCdos$ (this follows from Theorem~\ref{Comp_op_HCdos_necessity}). We begin by showing that a symbol as in \eqref{quinquennal} defines a composition operator on the space of double Dirichlet polynomials (finite series).

\begin{lemma} \label{comp_op_double_Dirichlet_polynomial}
Let $\phi: \C_+^2 \rightarrow \C_+^2$ be an analytic function and suppose there exists some $\sigma > 0$ such that $\phi_j(s,t)=c_j s + d_j t + \varphi_j(s,t)$ for $j=1,2$ and $(s,t) \in \C_\sigma^2$, where $\varphi_j(s,t)=\sum_{m,n=1}^\infty \frac{b^{(j)}_{m,n}}{m^s n^t}$ converges absolutely in $\C^2_\sigma$ and $c_j,d_j \in \N_0$, $j=1,2$. Then, if $D$ is a double Dirichlet polynomial, $D \circ \phi$ is a double Dirichlet series in $\mathcal{H}^\infty(\C_+^2)$.
\end{lemma}

\begin{proof}
We are going to see that $k^{-\varphi_1(s,t)}$ can be written as a double Dirichlet series in $\C_\sigma^2$. Using the expansion of the exponential function, if $(s,t) \in \C_\sigma^2$,
\begin{align*}
k^{-\varphi_1(s,t)} &=e^{- \log k \sum_{m,n=1}^\infty \frac{b^{(1)}_{m,n}}{m^s n^t} }\\
&=\prod_{m,n=1}^\infty e^{- \log k \frac{b^{(1)}_{m,n}}{m^s n^t}}\\
&=k^{-b_{1,1}^{(1)}}\prod_{\substack{m,n=1, \\ (m,n) \neq (1,1)}}^\infty \left ( 1 + \sum_{r=1}^\infty \frac{(-\log k \, b^{(1)}_{m,n})^r}{r!} \frac{1}{(m^r)^s} \frac{1}{(n^r)^t} \right ).
\end{align*}
Let us see how can we rearrange this expression for $k^{-\varphi_1(s,t)}$ into a double Dirichlet series. For each $(M,N) \in \N^2$, because of the absolute convergence, we may define $A_{k,M,N}$ in the following way:
\begin{enumerate}[(i)]
\item If $(M,N)=(1,1)$, then $A_{k,1,1} = 1$.
\item
If $M\neq 1$ and $N=1$, consider all possible factorizations of $M$ as $M=m_1^{r_1} \cdots m_d^{r_d}$ where $m_1,\dots,m_d\in\N \setminus \{1\}$ are all different, $r_1,\dots,r_d\in\N$ (there is at least one such factorization by setting $m_1=M$ and $r_1=1$). Now define
$$A_{k,M,1}= \sum_{m_1^{r_1} \cdots m_d^{r_d} = M} \left [ \prod_{j=1}^d \frac{(-\log k \; b_{m_j,1}^{(1)})^{r_j}}{r_j !} \right ].$$

\item If $M=1$, $N \neq 1$, proceeding analogously to the previous case, define
$$A_{k,1,N}= \sum_{n_1^{r_1} \cdots n_d^{r_d} = N} \left [ \prod_{j=1}^d \frac{(-\log k \; b_{1,n_j}^{(1)})^{r_j}}{r_j !} \right ].$$

\item If both $M,N \neq 1$, combining the two previous cases,
we define
$$A_{k,M,N} = \sum_{\substack{m_1^{r_1} \cdots m_d^{r_d} = M \\ n_1^{r_1} \cdots n_d^{r_d} = N}} \left [ \prod_{j=1}^d \frac{(-\log k \; b_{m_j,n_j}^{(1)})^{r_j}}{r_j !} \right ].$$
\end{enumerate}
Then for any $(s,t)\in\C_{\sigma}^2$,
$$k^{-\varphi_1(s,t)} = k^{-b_{1,1}^{(1)}}\sum_{M,N=1}^{\infty} \frac{A_{k,M,N}}{M^s N^t}.$$
With the same idea one gets
$$l^{-\varphi_2(s,t)} = l^{-b_{1,1}^{(2)}}\sum_{M,N=1}^{\infty} \frac{B_{l,M,N}}{M^s N^t}.$$
As these two double Dirichlet series are absolutely convergent in $\C_\sigma^2$, they can be multiplied to obtain 
$$k^{-\varphi_1(s,t)} l^{-\varphi_2(s,t)} = k^{-b_{1,1}^{(1)}}  l^{-b_{1,1}^{(2)}}\sum_{M,N=1}^\infty \frac{C_{k,l,M,N}}{M^s N^t}.$$
Finally, let $D(s,t)=\sum_{k=1}^K \sum_{l=1}^L \frac{a_{k,l}}{k^s l^t}$. Then, for $(s,t) \in \C_\sigma^2$,
\begin{align*}
(D\circ \phi)(s,t)  &= \sum_{k=1}^K \sum_{l=1}^L a_{k,l} \, k^{-c_1 s -d_1 t -b_{1,1}^{(1)}} \, l^{-c_2 s -d_2 t -b_{1,1}^{(2)} } \sum_{M,N=1}^\infty \frac{C_{k,l,M,N}}{M^s N^t} \\
&= \sum_{M,N=1}^\infty \sum_{k=1}^K \sum_{l=1}^L \frac{k^{-b_{1,1}^{(1)}} l^{-b_{1,1}^{(2)}}C_{k,l,M,N}}{(k^{c_1} l^{c_2} M)^s (k^{d_1} l^{d_2} N)^t}
\end{align*}
which can be rearranged into a double Dirichlet series which still is absolutely convergent on $\C_\sigma^2$. Moreover, as $\Vert D \circ \phi \Vert_\infty \leq \Vert D\Vert_\infty < \infty$, Bohr's theorem \cite[Theorem~2.7]{Nuestro} guarantees that $D \circ \phi \in \mathcal{H}^\infty(\C_+^2)$.
\end{proof}

Following the same scheme as in the one-dimensional case, some results concerning the range of the symbols of the composition operators are needed.

\begin{remark}\label{soldadet}
Suppose $\phi$ is an analytic function as in \eqref{quinquennal}, where the Dirichlet series $\varphi_{j}$ converge regularly on $\mathbb{C}_{\sigma}^{2}$ for some $\sigma >0$. Then the function $\tilde{\varphi}_{j} (s,t) = \phi(s,t) - c_{0}^{(j)} s - d_{0}^{(j)}t$, defined on $\mathbb{C}_{+}^{2}$, is clearly analytic and coincides with $\varphi_{j}$ on $\mathbb{C}_{\sigma}^{2}$. In other words, $\tilde{\varphi}_{j}$ is an analytic extension of $\varphi_{j}$ to $\mathbb{C}_{+}^{2}$. For the sake of clarity in the notation we will write $\varphi_{j}$ also for the extension, identifying the Dirichlet series with the extension. This is, for example, how the statement of Lemma~\ref{com_op_positive_range_varphi_double} should be understood.\\
On the other hand, if we suppose that  each $\varphi_{j}$ converges regularly on $\mathbb{C}_{+}^{2}$, then they define an analytic function (see \cite[page~531]{Nuestro}). Therefore if $\phi$ is as in \eqref{quinquennal}, then by Hartog's theorem, it is analytic. This is the case, for example, in Lemma~\ref{range_phi_comp_op_double} and Theorem~\ref{Composition_Operators_on_H_inf(C_+^2)_sufficiency}.
\end{remark}

\begin{lemma} \label{com_op_positive_range_varphi_double}
Suppose $\phi: \C_+^2 \rightarrow \C_+^2$ is an analytic function such that $\phi_j(s,t)=c_j s + d_j t + \varphi_j(s,t)$ for $j=1,2$, where $\varphi_j(s,t)=\sum_{m,n=1}^\infty \frac{b^{(j)}_{m,n}}{m^s n^t}$ converges in $\C^2_+$ and $c_j,d_j \in \N_0$, $j=1,2$. Then, $\re \varphi_j(s,t) \geq 0$ for all $(s,t) \in \C_+^2$, $j=1,2$.
\end{lemma}

\begin{proof}
Fix $s_0 \in \C_+$ and consider $\phi_j(s_0,t)=c_j s_0 + d_j t + \varphi_j(s_0,t) = d_j t + (c_j s_0 + \varphi_j(s_0,t))$. Using the first part of  \cite[Proposition~4.2]{Gordon_Hedenmalm99}, $\re (c_j s_0 + \varphi_j(s_0,t)) \geq 0$ for all $t \in \C_+$, but also for all $s_0 \in \C_+$.
Fixing now $t_0 \in \C_+$ and using again \cite[Proposition~4.2]{Gordon_Hedenmalm99}, $\re \varphi_j(s_0,t_0) \geq 0$ for all $(s_0, t_0) \in \C_+$.
\end{proof}

\begin{remark} \label{comp_op_double_range_re_varphi}
If $f: \C_+^2 \rightarrow \overline{\C_+}$ is a holomorphic function such that $\re f(s_0,t_0) = 0$ for some $(s_0,t_0) \in \C_+^2$, then in fact $ \re f(s,t) = 0$ for all $(s,t) \in \C_+^2$. Indeed, if we define
$f_{t_0}: \C_+ \rightarrow \overline{\C_+}$ as $f_{t_0}(s) = f(s,t_0)$ and suppose that $f_{t_0}$ is not constant, by the open mapping property $f_{t_0}(\C_+)$ is an open set, which contradicts the fact that $f(s_0,t_0) = i\tau_{0}$ for some $\tau_{0} \in \mathbb{R}$. Therefore, it is constant and $f_{t_0}(s) = i\tau_{0}$ for all $s \in \C_+$. 
Proceeding in the same way, defining for each $s \in \C_+$ a function $f_s: \C_+ \rightarrow \overline{\C_+}$ by $f_{s}(t) = f(s,t)$ we conclude that $f(s,t) = f_s(t) = i\tau_{0}$ for all $t \in \C_+$. This gives $ \re f \equiv 0$ in $\C_+^2$.\\
This allows to strengthen the Lemma~\ref{com_op_positive_range_varphi_double} to say that either $ \re \varphi_j(s,t) >0$ for all $(s,t) \in \C_+^2$, or $ \re \varphi_j(s,t) $ is constant and equal to zero. 
\end{remark}

We aim now at an analogue of Proposition~\ref{comp_op_range_phi_C_+} for double Dirichlet series. A fundamental tool is the following version in our setting of \cite[Theorem~3.3]{DeGaMaSe18}.
We could give a proof of it following the lines of \cite{DeGaMaSe18}, but we prefer to give a more elementary proof avoiding the sophisticated tools (Bohr transform,
Aron-Berner extension) used there.

\begin{lemma} \label{sup_MDS_strictly_decreasing}
Let $D(s,t)=\sum_{m,n=1}^\infty \frac{a_{m,n}}{m^s n^t}$ be a non-constant double Dirichlet series in $\mathcal{H}^\infty(\C_+^2)$, and $0< \sigma_1 < \eta_1$, $0 < \sigma_2 < \eta_2$. Then
$$ \sup_{\substack{\re s > \sigma_1 \\ \re t > \sigma_2}} \left | \sum_{m,n=1}^\infty \frac{a_{m,n}}{m^s n^t} \right | > \sup_{\substack{\re s > \eta_1 \\ \re t > \eta_2}} \left | \sum_{m,n=1}^\infty \frac{a_{m,n}}{m^s n^t} \right |.$$
\end{lemma}

\begin{proof}
We first recall that for all Dirichlet series $h$ belonging to $\HC$ and for all $\sigma>0$, then 
$\sup_{\re(s)=\sigma}|h(s)|=\sup_{\re(s)>\sigma}|h(s)|$ (see for instance \cite[Corollary~2.3]{DeGaMaSe18}). This yields
\begin{align*}
  \sup_{\substack{\re (s)>\alpha_1\\ \re(t)>\alpha_2}} |D(s,t)| &= \sup_{\re(t)>\alpha_2}\sup_{\re (s)>\alpha_1} |D(s,t)|\\
  &=\sup_{\re(t)>\alpha_2} \sup_{\re(s)=\alpha_1} |D(s,t)|\\
  &=\sup_{\re(s)=\alpha_1}\sup_{\re(t)>\alpha_2}  |D(s,t)|\\
  &=\sup_{\substack{\re (s)=\alpha_1\\ \re(t)=\alpha_2}}  |D(s,t)|.
\end{align*}
Observe also that, since $D$ is not constant, 
\[|a_{1,1}|<\sup_{\substack{\re (s)>\sigma_1\\ \re(t)>\sigma_2}} |D(s,t)| .\]
Again, this follows easily from the corresponding result in the one-dimensional case. Indeed, let $f(t)=\sum_{n=1}^{+\infty}a_{m,1}n^{-t}$ and $D_t(s)=D(s,t)$,
so that $f(t)$ is the constant term of the Dirichlet series $D_t$. If $f$ is constant, then there exists $t'$ with $\re(t')>\sigma_2$ such that 
$D_{t'}$ is not constant (otherwise $D$ itself would be constant). We then write
\[ |a_{1,1}| =f(2\sigma_2)<\sup_{\re(s)>\sigma_1} |D_{t'}(s)|\leq \sup_{\substack{\re (s)>\sigma_1\\ \re(t)>\sigma_2}} |D(s,t)| .\]
On the contrary, if $f$ is not constant, we write
\[ |a_{1,1}|<\sup_{\re(t)>\sigma_2} |f(t)|\leq \sup_{\substack{\re (s)>\sigma_1\\ \re(t)>\sigma_2}} |D(s,t)| .\]
Let $\theta_1,\theta_2\in (0,1)$ be such that $\eta_1=(1-\theta_1)\sigma_1+\theta_1 \gamma$,
$\eta_2=(1-\theta_2)\sigma_2+\theta_2\gamma$. Two successive applications of Hadamard's three lines theorem lead to
\begin{align*}
 \sup_{\substack{\re(s)=\eta_1\\ \re(t)=\eta_2}}|D(s,t)| 
 &\leq \sup_{\substack{\re (s)=\sigma_1\\ \re(t)=\sigma_2}}|D(s,t)|^{(1-\theta_1)(1-\theta_2)}\times \sup_{\substack{\re (s)=\sigma_1\\ \re(t)=\gamma}}|D(s,t)|^{(1-\theta_1)\theta_2} \\
&\quad\quad \times
\sup_{\substack{\re (s)=\gamma \\ \re(t)=\sigma_2}}|D(s,t)|^{\theta_1(1-\theta_2)}\times \sup_{\substack{\re (s)=\gamma \\ \re(t)=\gamma }}|D(s,t)|^{\theta_1\theta_2}\\
 &<\sup_{\substack{\re (s)=\sigma_1\\ \re(s)=\sigma_2}} |D(s,t)|. 
\end{align*}
\end{proof}

\begin{remark}
 The previous proof uses that if $D(s)=\sum_{n\geq 1}a_n n^{-s}$ is a nonconstant Dirichlet series converging in some half-plane $\CC_\sigma$, then $|a_1|<\sup_{\re(s)>\sigma}|D(s)|$.
 Although this is well known to specialists, we have not been able to locate a proof for this statement and we provide here an elementary one for the convenience of the reader.
 We may assume that $a_1$ is not equal to zero. 
 Let $k>1$ be such that $a_k\neq 0$ and $a_2,\dots,a_{k-1}$ are all equal to zero. Then as $\re (s)$ goes to $+\infty$, 
 $D(s)-a_1\sim a_k k^{-s}$. Let $\sigma_0>\sigma$ be such that $D(s)-a_1=(1+\e(s))a_k k^{-s}$ with $|\e(s)|<1/2$ for $\re (s)\geq\sigma_0$. Let $\tau\in\R$ be such that 
 $a_k k^{-\sigma_0-i\tau}=\lambda a_1 $ for some $a_1>0$. Then 
 \begin{align*}
  |D(\sigma_0+i\tau)&=\left|(1+\lambda)a_1+\e(s) a_kk^{-s}\right|\\
  &\geq \left(1+\frac\lambda2\right)|a_1|>|a_1|.
 \end{align*}
\end{remark}

\begin{lemma} \label{range_phi_comp_op_double}
Suppose $\phi: \C_+^2 \rightarrow \C_+^2$ is a function such that $\phi_j(s,t)=c_j s + d_j t + \varphi_j(s,t)$ for $j=1,2$, where $\varphi_j(s,t)=\sum_{m,n=1}^\infty \frac{b^{(j)}_{m,n}}{m^s n^t}$ converges regularly in $\C^2_+$ and $c_j,d_j \in \N_0$ for $j=1,2$. Then, for every $\e > 0$ there exists some $\delta > 0$ such that $\phi_j (\C_\e^2) \subset \C_\delta$ for all $j=1,2$.
\end{lemma}

\begin{proof}
Suppose $c_1 \neq 0$, that is $c_1 \in \N$. By Lemma~\ref{com_op_positive_range_varphi_double},
$\re \phi_1(s,t) = c_1 \re s + d_1 \re t + \re \varphi_1(s,t) > \e$ for $(s,t) \in \C_\e^2$, so for all $\e>0$, $\phi_1(\C_\e^2) \subset \C_\e$. The same argument applies in the case $d_1 \neq 0$. Now, if $c_1 = 0 =d_1$, then $\phi_1(s,t) = \varphi_1(s,t)$ for all $(s,t) \in \C_+^2$, so $\varphi_1 : \C_+^2 \rightarrow \C_+^2$ and $\re \varphi_1(s,t) > 0$ for all $(s,t) \in \C_+^2$. Now, if $\varphi_1$ is constant, then that constant has positive real part and the lemma is trivially satisfied. Otherwise we can apply Lemma~\ref{sup_MDS_strictly_decreasing} to $D(s,t)=2^{-\varphi_1(s,t)}$, which by Lemma~\ref{comp_op_double_Dirichlet_polynomial} is a non-constant double Dirichlet series. Therefore, given  $\e >0$,
$$\sup_{(s,t) \in \C_\e^2} |2^{-\varphi_1(s,t)}| = \sup_{(s,t) \in \C_\e^2} 2^{-\re \varphi_1(s,t)} < \sup_{(s,t) \in \C_+^2} 2^{-\re \varphi_1(s,t)} = \sup_{(s,t) \in \C_+^2} |2^{-\varphi_1(s,t)}|,$$
which implies that there exists some $\delta >0$ such that $$\inf_{(s,t) \in \C_\e^2} \re \varphi_1(s,t) > \delta > \inf_{(s,t) \in \C_+^2} \re \varphi_1(s,t) \geq 0,$$ that is, $\varphi_1(\C_\e^2) \subset \C_\delta$.
\end{proof}

\begin{remark} \label{rampell}
If $\sum_{m,n} a_{m,n} m^{-s} n^{-t}$ is a double Dirichlet series that converges at some $(s_{0}, t_{0}) \in \mathbb{C}^{2}$ then $\sup_{m,n} \big\vert \frac{a_{m,n}}{m^{s_{0}} n^{t_{0}}} \big\vert = K < \infty$ and
\[
\sum_{m,n=1}^{\infty} \Big\vert \frac{a_{m,n}}{m^{s_{0} + 1 + \varepsilon} n^{t_{0}+ 1 + \varepsilon}} \Big\vert \leq K \sum_{m=1}^{\infty} \frac{1}{m^{1+\varepsilon}} \sum_{n=1}^{\infty} \frac{1}{n^{1+\varepsilon}} < \infty \,.
\]
In other words, the Dirichlet series converges absolutely on $\mathbb{C}_{\re s_{0}+1} \times \mathbb{C}_{\re t_{0}+1}$.
\end{remark}

Now we are ready to prove that the symbols we are dealing with actually define composition operators, which is done in the following theorem. 

\begin{theorem} \label{Composition_Operators_on_H_inf(C_+^2)_sufficiency}
Let $ \phi = (\phi_1,\phi_2): \C_+^2 \rightarrow \C_+^2$ be a function such that $\phi_j(s,t)=c_j s + d_j t + \varphi_j(s,t)$ for $j=1,2$ and $(s,t) \in \C_+^2$, where $\varphi_j(s,t)=\sum_{m,n=1}^\infty \frac{b^{(j)}_{m,n}}{m^s n^t}$ converges regularly in $\C^2_+$ and $c_j,d_j \in \N_0$ for $j=1,2$. Then $\phi$ generates a composition operator $C_{\phi} : \mathcal{H}^\infty (\mathbb{C}^2_+) \rightarrow \mathcal{H}^\infty (\mathbb{C}^2_+)$. 
\end{theorem}

\begin{proof}
Take some $ D = \sum_{k,l=1}^\infty \frac{a_{k,l}}{k^s l^t}\in \mathcal{H}^\infty (\mathbb{C}^2_+)$ and let us see that $D \circ \phi \in \HCdos$. For each $(m,n)$ we denote $D_{m,n} =  \sum_{k,l=1}^m \sum_{l=1}^n \frac{a_{k,l}}{k^s l^t}$ (the partial sum of $D$) and use Lemma~\ref{comp_op_double_Dirichlet_polynomial} (note that by Remark~\ref{soldadet} $\phi$ is analytic, and by Remark~\ref{rampell} each $\varphi_{j}$ converges absolutely on $\mathbb{C}_{1}^{2}$) to have that $D_{m,n} \circ \phi$ is a double series in $\HCdos$, 
that we denote by $\sum_{k,l=1}^\infty \frac{c_{k,l}^{(m,n)}}{k^s l^t}$. \\
Now, by Lemma~\ref{range_phi_comp_op_double}, given $\e > 0$ there exists $\delta > 0$ such that $\phi(\C^2_\e) \subset \C^2_\delta$. As a consequence of \cite[Theorem~2.7]{Nuestro},  the partial sums $D_{m,n}$ are uniformly convergent on $\C^2_\delta$  to $D$. Therefore the sequence $\{D_{m,n} \circ \phi\}_{m,n}$ is uniformly bounded on $\C_\e^{2}$, say by $C$. This implies that
the horizontal translates defined by $\sum_{k,l=1}^\infty \frac{c_{k,l}^{(m,n)}}{k^{s+\e} l^{t+\e}}$ are in $\HCdos$ for every $m,n \in \N$. Applying now \cite[Proposition~2.2]{Nuestro}, which controls the coefficients of a series in $\HCdos$ by its norm, we get that for every $k,l,m,n \in \N$,
\[
\frac{|c_{k,l}^{(m,n)}|}{k^\e l^\e} 
\leq \sup_{(s,t) \in \C_\e^2} \vert D_{m,n}(\phi(s,t))  \vert 
\leq \sup_{(s,t) \in \C_\delta^2}  \vert D_{m,n}(s,t)  \vert  \leq C.
\]
Therefore for fixed $k$ and $l$ we have
\[
|c_{k,l}^{(m_1,n_1)} - c_{k,l}^{(m_2,n_2)}| 
\leq k^\e l^\e \sup_{(s,t) \in \C_\delta^2} \vert D_{m_1,n_1} (s,t) - D_{m_2,n_2} (s,t)\vert, 
\]
so the double sequence $\{c_{k,l}^{(m,n)}\}_{m,n}$ converges to some $c_{k,l} \in \C$ satisfying $|c_{k,l}| \leq C_\sigma k^\sigma l^\sigma$ for all $\sigma > \e$.
Define now $F(s,t)=\sum_{k,l=1}^\infty \frac{c_{k,l}}{k^s l^t}$. Since $|c_{k,l}| \leq C (kl)^{2\e}$, $F(s,t)$ converges absolutely in $\C_{ 1 + \e}^2$, and there the double sequence $\{D_{m,n} \circ \phi\}_{m,n}$ clearly converges absolutely to $F$. It is enough to note that $\Vert D \circ \phi\Vert_\infty \leq \Vert D \Vert_\infty < \infty$ to apply \cite[Theorem~2.7]{Nuestro} and get that $D \circ \phi$ actually coincides with $F$ and that it is in $\mathcal{H}^\infty(\C_+^2)$.
\end{proof}

\subsection{The necessary condition}

Theorem~\ref{Composition_Operators_on_H_inf(C_+^2)_sufficiency} gives the sufficient condition for the characterization of the composition operators of $\HCdos$ in Theorem~\ref{rural}. To prove the necessity we use the vector-valued perspective introduced in \cite{Nuestro} to deal with double Dirichlet series (formalized in Lemma~\ref{comp_op_fixing_variable_symbol}). But before that let us first recall some notation. Given $\sum_{m,n} \frac{a_{m,n}}{m^s n^t} \in \HCdos$, for each $m \in \mathbb{N}$ we define the row subseries $\alpha_m(t) = \sum_{n=1}^\infty \frac{a_{m,n}}{n^t}  \in \HC$. Then $D(s,t)= \sum_{m=1}^\infty \frac{\alpha_m(t)}{m^s}$ for every $(s,t) \in \C_+^2$.

\begin{lemma} \label{comp_op_fixing_variable_symbol}
Let $\phi= (\phi_1,\phi_2): \C_+^2 \rightarrow \C_+^2$ be inducing a composition operator $C_\phi : \HCdos \rightarrow \HCdos$.
For each fixed $t \in \C_+$ and $j=1,2$, consider $\phi_{j,t}: \C_+ \rightarrow \C_+$ given by $\phi_{j,t}(s)= \phi_j(s,t)$. Then $\phi_{j,t}$ defines a composition operator of $\HC$.
\end{lemma}

\begin{proof}
We just deal with the case $j=1$, the other one being analogous. Take $D(s) = \sum_{m=1}^\infty \frac{a_m}{m^s} \in \HC$, and define $b_{m,1} = a_m$, $b_{m,n}=0$ for $n \geq 2$, and $\tilde{D}(s,t) = \sum_{m,n=1}^\infty \frac{b_{m,n}}{m^s n^t}$. Clearly $\tilde{D}(s,t) = D(s)$ for every $(s,t) \in \C_+^2$, and $\tilde{D} \in \HCdos$, so $\tilde{D} \circ \phi \in \HCdos$. Now,
$$(\tilde{D} \circ \phi) (s,t)= \sum_{m,n=1}^\infty \frac{b_{m,n}}{m^{\phi_1(s,t)} n^{\phi_2(s,t)}} = \sum_{m=1}^\infty \frac{a_m}{m^{\phi_1(s,t)}}.$$
Then, for a fixed $t \in \C_+$,
$$(D \circ \phi_{1,t})(s) = \sum_{m=1}^\infty \frac{a_m}{m^{\phi_{1,t}(s)}} = \sum_{m=1}^\infty \frac{a_m}{m^{\phi_1(s,t)}}=(\tilde{D} \circ \phi) (s,t),$$
so $D \circ \phi \in \HC$. 
\end{proof}

We still need a further lemma before we give the main step towards the necessity in Theorem~\ref{rural}.

\begin{lemma} \label{comp_op_Dirichlet_series_over_semigroups}
Consider $\varphi_2(s) = \sum_n \frac{a_n}{n^s}$ and $\varphi_3(s)=\sum_m \frac{b_m}{m^s}$ two Dirichlet series that converge absolutely in $\C_\sigma$ and 
let $\varphi$ be any function defined on $\C_\sigma$. If there exists $c_0 \in \N$ such that
\[
\sum_{n=1}^\infty a_n \left ( \frac{2^{c_0}}{n} \right )^s = 2^{c_0 s} \varphi_2(s) = \varphi(s) = 3^{c_0 s} \varphi_3(s) = \sum_{m=1}^\infty b_m \left ( \frac{3^{c_0}}{m} \right )^s\,
\]
for all $s \in \C_\sigma$, then $\varphi$ is also a Dirichlet series that converges absolutely in $\C_\sigma$.
\end{lemma}

\begin{proof}
Let $j \in \N$ such that it is not a multiple of $2^{c_0}$. Then, using \cite[Proposition~1.9]{DeGaMaSe})
\begin{multline*}
a_j  = \lim_{T \rightarrow \infty} \frac{1}{2T} \int_{\sigma +1 -iT}^{\sigma +1 +iT} 2^{c_0 s} \varphi_2(s) \left (\frac{j}{2^{c_0}} \right )^s ds  
 = \lim_{T \rightarrow \infty} \frac{1}{2T} \int_{\sigma +1 -iT}^{\sigma +1 +iT} 3^{c_0 s} \varphi_3(s) \left (\frac{j}{2^{c_0}} \right )^s ds  \\ 
= \lim_{T \rightarrow \infty} \frac{1}{2T} \int_{\sigma +1 -iT}^{\sigma +1 +iT} \sum_{m=1}^\infty b_m \left( \frac{3^{c_0}}{m} \right )^s \left (\frac{j}{2^{c_0}} \right )^s ds 
 = \sum_{m=1}^\infty b_m \lim_{T \rightarrow \infty} \frac{1}{2T} \int_{\sigma +1 -iT}^{\sigma +1 +iT} \left( \frac{3^{c_0}}{2^{c_0}} \frac{j}{m} \right )^s ds = 0,
\end{multline*}
because $\dfrac{{3^{c_0}j}}{2^{c_0}m}$ is not an integer. Hence, all the coefficients of $\varphi_2$ corresponding to non-multiples of $2^{c_0}$ are null, so 
$$\varphi(s) = \sum_{n=1}^\infty a_n \left ( \frac{2^{c_0}}{n} \right )^s = \sum_{j=1}^\infty a_{j2^{c_0}} \left ( \frac{2^{c_0}}{j2^{c_0}} \right )^s = \sum_{j=1}^\infty \frac{a_{j2^{c_0}}}{j^s}.$$
Therefore $\varphi$ is a Dirichlet series which converges absolutely in $\C_\sigma$.
\end{proof}

\begin{theorem} \label{Comp_op_HCdos_necessity}
Let $\phi=(\phi_1, \phi_2): \C_+^2 \rightarrow \C_+^2$ be inducing a composition operator $C_\phi : \HCdos \rightarrow \HCdos$. Then there exists some $\sigma > 0$ such that, for $j=1,2$, $\phi_j(s,t) = c_0^{(j)}s + d_0^{(j)} t + \varphi_j(s,t)$ for $(s,t) \in \C_\sigma^2$, where $c_0^{(j)}, d_0^{(j)} \in \N_0$ and $\varphi_j(s,t) = \sum_{m,n=1}^\infty \frac{b_{m,n}^{(j)}}{m^s n^t}$ is a double Dirichlet series that converges absolutely in $\C_\sigma^2$.
\end{theorem}

\begin{proof}
The proof works for $j=1$ or $j=2$ so, to keep the notation simpler, we will drop the subscript and consider $\phi(s,t): \C_+^2 \rightarrow \C_+$. On the one hand, by hypothesis $D_k(s,t):= k^{-\phi(s,t)} \in \HCdos$ for every $k \in \N$. Using the regular convergence of $D_k$, if $t \in \C_+$ is fixed,
\begin{equation} \label{serotonina}
k^{-\phi(s,t)} = \sum_{m,n=1}^\infty \frac{a_{m,n}^{(k)}}{m^s n^t} = \sum_{m=1}^\infty \frac{1}{m^s} \sum_{n=1}^\infty \frac{a_{m,n}^{(k)}}{n^t} = \sum_{m=1}^\infty \frac{\alpha^{(k)}_m(t)}{m^s} = D_{k,t}(s).
\end{equation}
On the other hand, for $t \in \C_+$ still fixed, Lemma~\ref{comp_op_fixing_variable_symbol} gives that $\phi_t$ defines a composition operator of $\HC$, so by Theorem~\ref{Composition_Operators_on_H_inf(C_+)_FINAL} $\phi_t(s) = c_0(t) s + \varphi_t(s)$, where $c_0(t) \in \N_0$ and $\varphi_t(s) = \sum_{m=1}^\infty \frac{c_m(t)}{m^s}$ is a Dirichlet series that converges regularly and uniformly in the half-plane $\C_\e$ for every $\e > 0$. 

We first show that $t\mapsto c_0(t)$ is constant on some half-plane $\overline{\C_{\sigma_0}}$. Using the  argument in \cite[page~316]{Gordon_Hedenmalm99}
\[
k^{c_0(t)}  = \inf\big(\{m\in\N\ :\ \alpha_m^{(k)}(t)\neq 0\}\big).
\]
This means that the series in \eqref{serotonina} actually runs up for $m \geq k^{c_{0}}$ and
\[
c_0(t)=\inf\big(\{p\in\N\ :\ \alpha_{k^p}^{(k)}(t)\neq 0\}\big).
\]
Define
$$c_0=\inf\big(\{p\in\N\ :\ \alpha_{k^p}^{(k)}\textrm{ is not identically zero}\}\big).$$
Let $a_{k^{c_0},N}^{(k)}$ be the first non-zero coefficient of $\alpha_{k^{c_0}}^{(k)}$. By  \cite[Lemma~3.1]{Gordon_Hedenmalm99} we can find $\sigma_0 > 0$ such that 
\[
\sup_{t \in \overline{\C_{\sigma_0}}} |N^t \alpha_{k^{c_0}}^{(k)}(t) - a_{k^{c_0},N}^{(k)}| \leq \dfrac{|a_{k^{c_0},N}^{(k)}|}{2}.
\]
Hence $\alpha_{k^{c_0}}^{(k)}$ has no zeros in the half-plane $\overline{\C}_{\sigma_0}$ and therefore $c_0(t) = c_0$ for every $t \in \overline{\C}_{\sigma_0}$. \\

Since for each $t$ the Dirichlet series $\varphi_{t}$ converges uniformly on every half-plane strictly contained in $\mathbb{C}_{+}$,
\cite[Theorem~4.4.2]{Queffelec13} (see also \cite[Proposition~1.10]{DeGaMaSe}) implies that $\varphi_t(s)$ is absolutely convergent for every $s$ with $\re s > 1/2$. Take, then, $(s,t) \in \C_{\frac{1}{2}} \times \C_{\sigma_0}$. Following \cite[Theorem~A]{Gordon_Hedenmalm99} and proceeding as in Lemma~\ref{comp_op_double_Dirichlet_polynomial} (using the Taylor expansion of the exponential) we arrive at
\begin{equation} \label{sabotatge}
\sum_{m=1}^\infty \frac{\alpha^{(k)}_m(t)}{m^s} = k^{-\phi(s,t)} = k^{-c_0 s} k^{-c_1(t)} \prod_{m=2}^\infty \left ( 1+ \sum_{j=1}^\infty \frac{(-1)^{-j} (c_m(t))^j (\log k)^j}{j! (m^j)^s} \right ) \,.
\end{equation}
Expanding the product at the right-hand side yields a series whose terms we can rearrange (because all the involved series are absolutely convergent), into a Dirichlet series the coefficients of which we denote by  $d_l^{(k)}(t)$. Hence
\begin{equation} \label{mamimiraamunt}
\sum_{m=1}^\infty \frac{\alpha^{(k)}_m(t)}{m^s} 
= \frac{k^{-c_1(t)}}{k^{c_0 s}} \left( 1+ \sum_{l=2}^\infty \frac{d_l^{(k)}(t)}{l^s} \right)
= \frac{k^{-c_1(t)}}{k^{c_0 s}} + \sum_{l=2}^\infty \frac{k^{-c_1(t)} d_l^{(k)}(t)}{(lk^{c_0})^s}.
\end{equation}
So, we have arrived at an equality between Dirichlet series that converge absolutely on some half-plane and we may identify coefficients (recall that we already saw that the series in the left-hand side in fact starts at $k^{c_{0}}$). To begin with, $k^{-c_1(t)}$ is the coefficient
corresponding to the term $(k^{c_0})^s$, so $k^{-c_1(t)} = \alpha_{k^{c_0}}^{(k)}(t)$ for all $t \in \C_{\sigma_0}$, and $\alpha_{k^{c_0}}^{(k)} \in \HC$. Since this holds for every  $k$, we can apply
Lemma~\ref{Comp_Op_monomials} to get that $c_1(t)$ is holomorphic in $\mathbb{C}_{\sigma_0}$ and that there exists some $\sigma_1 \geq \sigma_0$ such that $c_1(t) = d_0 t + \sum_{n=1}^\infty \frac{b_{1,n}}{n^t}$ for every $t \in \C_{\sigma_1}$, with $\sum_{n=1}^\infty \frac{b_{1,n}}{n^t}$ absolutely convergent in $\C_{\sigma_1}$.\\

What we want to do now is to push further this idea, comparing coefficients in \eqref{mamimiraamunt} in a systematic way to end up showing that every $c_m(t)$ can be written as a Dirichlet series absolutely convergent in $\C_{\sigma_1}$. We do this by induction on $m \geq 2$ and start with the case $m=2$. \\
We take some $(s,t) \in \C_{\frac{1}{2}} \times \C_{\sigma_1}$ and note that the term corresponding to $l=2$ in \eqref{mamimiraamunt} is obtained by multiplying the term $m=2$ and $j=1$ (this carries $c_{2}(t)$) and $1$'s in \eqref{sabotatge}. In this way we have
\[
\sum_{m=1}^\infty \frac{\alpha^{(k)}_m(t)}{m^s}
= \frac{k^{-c_1(t)}}{k^{c_0 s}} + \frac{-\log k \, c_2(t) k^{-c_1(t)}}{(2k^c_0)^s} + \sum_{l=3}^\infty \frac{k^{-c_1(t)} d_l^{(k)}(t)}{(lk^{c_0})^s}.
\]
Identifying again coefficients we get $\alpha_{2k^{c_0}}^{(k)}(t) = -\log k \, c_2(t) k^{-c_1(t)}$, so \[
c_2(t) = \frac{-1}{\log k} k^{c_1(t)} \alpha_{2k^{c_0}}^{(k)} = k^{d_0 t} \frac{-1}{\log k} k^{-(-\sum_{n=1}^\infty \frac{b_{1,n}}{n^t})} \alpha_{2k^{c_0}}^{(k)} = k^{d_0 t} \psi_k(t)\,.
\]
We need to see now that $\psi_k(t)$ is a Dirichlet series that converges absolutely in $\C_{\sigma_1}$. Note first that $\alpha_{2k^{c_0}}^{(k)}$ belongs to  $\HC$. 
On the other hand,  we have just seen that $-\sum_{n=1}^\infty b_{1,n} n^{-t}$ is an absolutely convergent series in $\C_{\sigma_1}$, and a
careful  inspection of the proof of the sufficiency of \cite[Theorem~A]{Gordon_Hedenmalm99} shows that $\displaystyle{k^{\sum_{n=1}^\infty \frac{b_{1,n}}{n^t}}}$ is an absolutely convergent Dirichlet series in $\C_{\sigma_1}$. This gives the claim. Letting now $k=2,3$ we have $ 2^{c_0 t} \psi_2(t) = c_2(t) = 3^{c_0 t} \psi_3(t)$ for every $t \in \C_{\sigma_1}$. 
Since $c_{2}(t)$ is analytic,
Lemma~\ref{comp_op_Dirichlet_series_over_semigroups} gives that $c_2(t) = \sum_{n=1}^\infty \frac{b_{2,n}}{n^t}$ is a Dirichlet series that converges absolutely in $\C_{\sigma_1}$. 
This completes the proof of the fact for $m=2$.

Suppose now that $c_{m}(t)$ is analytic for every $2 \leq m \leq m_{0}$. We want to use again \eqref{mamimiraamunt}, comparing the coefficients of the term corresponding to $l=m_{0}$. Note that we get this factor by multiplying the term $m=m_{0}$ and $j=1$ with all $1$'s (this brings $c_{m_{0}}(t)$) and the product of terms involving divisors of $m_{0}$ (this brings other $c_{m}(t)$'s, that we group in a term $D_{m_{0}}$). Let us be more precise. Starting from \eqref{sabotatge} we get
\[
\sum_{m=1}^\infty \frac{\alpha^{(k)}_m(t)}{m^s} 
 =  \frac{-\log k \, c_{m_{0}}(t) k^{-c_1(t)} + D_{m_{0}(t)}}{(m_{0}k^{c_0})^s} + \sum_{l \neq m_{0}}^\infty \frac{k^{-c_1(t)} d_l^{(k)}(t)}{(lk^{c_0})^s},
\]
where $D_{m_{0}}$ is given by
\[
D_{m_{0}}(t) = \sum_{q>1} \sum_{m_1^{r_1} \cdots m_q^{r_q}=m_{0}}  \prod_{h=1}^q \frac{\big(- \log k \, c_{m_h}(t) \big)^{r_h}}{r_h !}.
\]
Since $m_1^{r_1} \cdots m_q^{r_q} = m_{0}$ for $q>1$ implies that $m_h < m_{0}$ for all $1 \leq h \leq q$ we have $D_{m_{0}}$ is a finite sum of finite products of Dirichlet series which by the induction hypothesis are absolutely convergent in $\C_{\sigma_1}$. Hence $D_{m_{0}}$ is a Dirichlet series that converges absolutely on $\C_{\sigma_1}$. Then 
\[
c_{m_{0}}(t) = \frac{-k^{-c_1(t)}}{\log k}(\alpha_{m_{0}k^{c_0}}^{(k)} - D_m(t))
= k^{c_0 t} \frac{-1}{\log k}k^{(\sum_{n=1}^\infty \frac{b_{1,n}}{n^t})}(\alpha_{m_{0}k^{c_0}}^{(k)} - D_{m_{0}}(t))
=k^{c_0 t} \psi^{(m_{0})}_k(t),
\]
where, with the same argument as above, $\psi^{(m_{0})}_k$ is again an absolutely convergent Dirichlet series on $\C_{\sigma_1}$. Once again by application of Lemma~\ref{comp_op_Dirichlet_series_over_semigroups}, $c_{m_{0}}(t) = \sum_{n=1}^\infty \frac{b_{m_{0},n}}{n^t}$ is  a Dirichlet series that converges absolutely on $\C_{\sigma_1}$.\\

Finally, for $(s,t) \in \C_{\frac{1}{2}} \times \C_{\sigma_1}$,
$$\phi(s,t) = c_0 s + d_0 t + \sum_{m=1}^\infty \frac{1}{m^s} \sum_{n=1}^\infty \frac{b_{m,n}}{n^t} = c_0 s + d_0 t +  \sum_{m,n=1}^\infty \frac{b_{m,n}}{m^s n^t},$$
where the last equality holds because the sums converge absolutely in $\C_{\frac{1}{2}} \times \C_{\sigma_1}$.
\end{proof}

Once we have established the form of the symbols of composition operators we may proceed as in \cite{Queffelec15} to strengthen the conditions on the symbol in terms of its uniform convergence in $\C_\e^2$ for every $\e > 0$, giving the proof of our main result.

\begin{proof}[Proof of Theorem~\ref{rural}]
To begin with, if $\phi_{j}$ is as in \eqref{quinquennal} and $\varphi_{j}$ converges uniformly and regularly on $\mathbb{C}_{\varepsilon}$ for every $\varepsilon >0$, then it converges regularly on
$\mathbb{C}_{+}$ and Theorem~\ref{Composition_Operators_on_H_inf(C_+^2)_sufficiency} gives that $C_{\phi}$ defines a composition operator on $\HCdos$.\\
For the necessary condition, Theorem~\ref{Comp_op_HCdos_necessity} gives that each $\varphi_{j}$ converges absolutely on $\mathbb{C}_{\sigma}^{2}$ for some $\sigma >0$.
We adapt the arguments of \cite[Section~3]{Queffelec15} to see that in fact they converge uniformly on $\mathbb{C}_{\varepsilon}^2$
(we do it only for $j=1$). Note first that Lemma~\ref{comp_op_fixing_variable_symbol}, Lemma~\ref{Comp_Op_monomials} and Hartog's theorem give that $\phi$ is analytic. 
Then, by Lemma~\ref{com_op_positive_range_varphi_double} (recall also Remark~\ref{soldadet}), we get $\varphi_1(\C_+^2) \subset \overline{\C_+}$, that is, $|\Arg \varphi_1 | \leq \frac{\pi}{2}$ and  $|\Arg \varphi_1^{1/2} | \leq \frac{\pi}{4}$, from which $\left |\frac{\im \varphi_1^{1/2}}{\re \varphi_1^{1/2}}\right | =  |\tan(\Arg \varphi_1^{1/2}) | \leq 1$ follows. \\
We consider now the function $u(s,t) = \re (\varphi_1(s,t)^{1/2})$. Since $\varphi_{1}$ converges absolutely on $\mathbb{C}_{\sigma}^{2}$ it is uniformly bounded there. Now it is easy to see that $|\re \varphi_1^{1/2}| \leq |\re \varphi_1|^{1/2} \leq |\varphi_1|^{1/2} \leq \sqrt{2}|\re \varphi_1^{1/2}|$, because $|\varphi_1^{1/2}| = \sqrt{ \re \varphi_1 + \im \varphi_1} \leq \sqrt{2 \re \varphi_1}$, and then $u$ is uniformly bounded on $\overline{\C_\theta^2}$, say by $K$. Our aim now is to see that, in fact, $u$ is uniformly bounded on $\mathbb{C}_{\varepsilon}$ for every $\varepsilon >0$. By Remark~\ref{comp_op_double_range_re_varphi} either $\varphi_{1}$ is identically zero or $\re \varphi_{1}(s,t) >0$ for every $(s,t) \in \mathbb{C}_{+}^{2}$. If it is identically zero, then so also is $\varphi_{1}^{1/2}$ and therefore $u$. If this is not the case, then (since $|\Arg \varphi_1^{1/2} | \leq \frac{\pi}{4}$) $u$ is strictly positive.\\
If $u$ is identically zero then the claim is trivially satisfied. We may then assume that $u$ is positive and take $(s_{0},t_{0}) =(\sigma_1 + i\tau_1, \sigma_2 + i\tau_2) \in \C_\e^2$. We know from \cite[Section~3]{Queffelec15} that if $v$ is a positive harmonic function defined on some $\mathbb{C}_{\sigma_{0}}$, then
\begin{equation} \label{harnack}
v(\theta_{1} + i \tau) \leq \frac{\theta_{2}}{\theta_{1} }v(\theta_{2} + i \tau)
\end{equation}
for every $\sigma_0 < \theta_{1} \leq \theta_{2}$. Suppose that $\varepsilon< \sigma_{1} < \sigma$ and consider $u_{t_{0}} (s) = u(s,t_{0}) = \re  (\varphi_{1,t_{0}}(s)^{1/2})$. Suppose now that $\varepsilon < \sigma_{1} < \sigma$ then, since $u_{t_{0}}$ is a positive harmonic function \eqref{harnack} gives
\begin{equation} \label{kapsberger}
u(s_{0}, t_{0}) = u_{t_{0}} (\sigma_{1} + i \tau_{1})
\leq \frac{\sigma}{\sigma_1} u_t(\sigma + i\tau_1) \leq \frac{\sigma}{\e}u(\sigma + i\tau_1, \sigma_2 + i\tau_2) \,.
\end{equation}
We distinguish two cases for $t_{0}$. First, if $\sigma \leq \sigma_{2}$ we immediately obtain (recall that $u$ is uniformly bounded on $\overline{\mathbb{C}_{\sigma}^{2}}$ by $K$)
\[
u(s_{0}, t_{0}) \leq \frac{\sigma}{\e} K \,.
\]
On the other hand, if $\e < \sigma_2 < \sigma$, we can consider $\tilde{s} = \theta + i\tau_1 \in \overline{\C_\theta}$ and $t\mapsto u_{\tilde{s}}(t)$, which is again a positive harmonic function.
Starting with \eqref{kapsberger} and using \eqref{harnack} again (this time for $u_{\tilde{s}}$) we get
\[
u(s_{0}, t_{0}) \leq  \frac{\sigma}{\e} u_{\tilde{s}} (\sigma_2 + i\tau_2)
\leq  \frac{\sigma^2}{\e^2} u_{\tilde{s}}(\sigma + i\tau_2) \leq \frac{\sigma^2}{\e^2} K.
\]
The only case left to check is that in which $\e < \sigma_2 < \sigma$ and $\sigma_1 \geq \sigma$, but it is completely analogous to the one above, so we get $u(s,t) \leq \frac{\theta^2}{\e^2} K$ for every $(s,t) \in \C_\e$. Hence,
$\varphi_1^{1/2}$, and therefore $\varphi_1$, is uniformly bounded in $\C_\e$ and consequently uniformly convergent.
\end{proof}

The characterization of compact composition operators on $\HC$ given in \cite[Theorem~18]{Bayart02} still works for double Dirichlet series.

\begin{theorem}
Let $\phi$ define a continuous composition operator $C_{\phi} : \mathcal{H}^\infty (\mathbb{C}^2_+) \rightarrow \mathcal{H}^\infty (\mathbb{C}^2_+)$. Then $C_{\phi}$ is compact if and only if $\phi(\C_+^2) \subset \C_\delta^2$ for some $\delta >0$.
\end{theorem}

\begin{proof}
First, suppose that there exists some $\delta >0$ such that $\phi(\C_+^2) \subset \C_\delta^2$ and consider $\{D_n\}$ a bounded sequence in $\HCdos$. By \cite[Lemma~3.4]{Nuestro}, there exists some subsequence $\{D_{n_k} \}$ in $\HCdos$ and $D \in \HCdos$ such that $D_{n_k}$ converges uniformly to $D$ on $\C_\delta^2$, and therefore $D_{n_k} \circ \phi$ converges to $D \circ \phi$ uniformly on $\C_+^2$. Hence, $C_\phi$ is compact. On the other hand, if $C_\phi$ is compact, choose the sequence $D_m(u,v)=m^{-u}$ in $\HCdos$. Since $C_\phi$ is compact there exists some subsequence $D_{m_k}$ in $\HCdos$ and some $\tilde{D} \in \HCdos$ such that $\lim_{k \rightarrow \infty} \Vert D_{m_k} \circ \phi - \tilde{D} \Vert_\infty =0$. Fix $(s,t) \in \C_+^2$, then $\tilde{D}(s,t) = \lim_{k \rightarrow \infty} m_k^{- \phi_1(s,t)} = 0$ since $\re \phi_1(s,t) > 0$. Therefore
\[
0 = \lim_{k \rightarrow \infty} \Vert D_{m_k} \circ \phi - \tilde{D} \Vert_\infty 
= \lim_{k \rightarrow \infty} \Vert D_{m_k} \circ \phi \Vert_\infty =\lim_{k \rightarrow \infty} m_k^{- \inf_{(s,t) \in \C_+^2} \re \phi_1(s,t)},
\]
so necessarily $\inf_{(s,t) \in \C_+^2} \re \phi_1(s,t) > 0$ and there exists some $\delta_1 > 0$ such that $\phi_1(\C_+^2) \subset \C_{\delta_1}$. Applying the same idea to the sequence of functions $n^{-v}$ in $\HCdos$, one gets that there exists some $\delta_2 >0$ such that $\phi_2(\C_+^2) \subset \C_{\delta_2}$. Taking $\delta = \min (\delta_1, \delta_2)$ we have $\phi(\C_+^2) \subset \C_\delta^2$.
\end{proof}

\subsection{Composition operators on $\HCdos$ and on $H^\infty(B_{c_0^2})$}

We finish this section by relating composition operators of $\HCdos$ and composition operators on the corresponding space of holomorphic functions.
Observe that this result is new even for the one dimensional case (this is Proposition~\ref{Comp_op_H_infty_B_c_0}) and that it is more difficult 
than the corresponding result for $\mathcal H^p$ and $H^p(\T^\infty)$, $1\leq p<+\infty$ obtained in \cite{Bayart02}. Indeed, for finite values of $p$,
Dirichlet polynomials are dense in $\mathcal H^p$ and this was a key point in the proof done in \cite{Bayart02}.

\begin{proposition} \label{Proposition_comp_ops_HCdos_H(B_c_0dos)}
Let $\phi$ define a continuous composition operator $C_{\phi} : \mathcal{H}^\infty (\mathbb{C}^2_+) \rightarrow \mathcal{H}^\infty (\mathbb{C}^2_+)$. Then $C_\phi$ induces a composition operator $C_\psi: H^\infty(B_{c_0^2}) \rightarrow H^\infty(B_{c_0^2})$.
\end{proposition}

\begin{proof}
Let us recall how the bijective isometry $\mathcal{B}: H^\infty(B_{c_0^2}) \rightarrow \HCdos$ from  \cite[Theorem~3.5]{Nuestro}  (now with $k=2$) is defined. For each $f \in H^\infty(B_{c_0^2}$
with coefficients $c_{\alpha, \beta}(f)$ (that can be computed through the Cauchy integral formula) we have
\[
\mathcal B (f)
= \sum_{\alpha,\beta} \frac{c_{\alpha,\beta}(f)}{(p^\alpha)^{s}(p^\beta)^{t}}.
\]
Clearly $T= \mathcal{B}^{-1} \circ C_\phi \circ \mathcal{B}$ is an operator, $T: H^\infty(B_{c_0^2}) \rightarrow H^\infty(B_{c_0^2})$, and our aim is to see that it is actually a composition operator. For $j \in \N$ define $D^{(j)}_1 = C_{\phi} (\frac{1}{p_j^s}) = p_j^{-\phi_{1}(s,t)} \in \HCdos$ and $D^{(j)}_2 = C_{\phi} (\frac{1}{p_j^t}) = p_j^{-\phi_{2}(s,t)}\in \HCdos$ and $D^{(j)}=(D_1^{(j)}, D_2^{(j)})$. Note that $\mathcal{B}^{-1}(D_1^{(j)}), \mathcal{B}^{-1}(D_2^{(j)}) \in H^\infty(B_{c_0^2})$ for every $j \in \N$. Define formally $\Phi = (\Phi_{1}, \Phi_{2})=( (D_1^{(j)})_j, (D_2^{(j)})_j)$ and $\psi= (\psi_1, \psi_2)=((\mathcal{B}^{-1}D_1^{(j)})_j, (\mathcal{B}^{-1}D_2^{(j)})_j)$.
Then, if we consider a polynomial $f(z, \omega)= \sum_{\alpha, \beta \in \Lambda} c_{\alpha, \beta} z^\alpha \omega^\beta$ with $\Lambda$ finite and we denote by $\mathfrak{p}$ the sequence of prime numbers, we get
\begin{equation} \label{eq:T_and_C_psi_coincide_on_finite_polynomials_double}
\begin{split}
T(f)(z, \omega) & =\mathcal{B}^{-1} \left ( C_\phi \left ( \mathcal{B} \left (\sum_{\alpha, \beta \in \Lambda} c_{\alpha, \beta} z^\alpha \omega^\beta \right ) \right ) \right )  = \mathcal{B}^{-1} \left ( C_\phi \left (  \sum_{\alpha, \beta \in \Lambda} \frac{c_{\alpha,\beta}}{(\mathfrak{p}^\alpha)^s (\mathfrak{p}^\beta)^t} \right ) \right ) \\
& = \mathcal{B}^{-1} \left ( \sum_{\alpha, \beta \in \Lambda} \frac{c_{\alpha, \beta}}{(\mathfrak{p}^{\phi_1(s,t)})^\alpha (\mathfrak{p}^{\phi_2(s,t)})^\beta} \right ) = \mathcal{B}^{-1} \left ( \sum_{\alpha, \beta \in \Lambda} c_{\alpha, \beta} \Phi_1(s,t)^\alpha \Phi_2(s,t)^\beta \right ) \\ 
& =\sum_{\alpha, \beta \in \Lambda} c_{\alpha, \beta}( \mathcal{B}^{-1} \Phi_1)(z,\omega)^\alpha ( \mathcal{B}^{-1} \Phi_2)(z,\omega)^\beta  =\sum_{\alpha, \beta \in \Lambda} c_{\alpha, \beta} \psi_1(z,\omega)^\alpha \psi_2(z,\omega)^\beta .
\end{split}
\end{equation}
Therefore, $T$ coincides with the composition operator $C_\psi$ on finite polynomials, and we will see that they actually are the same operator. 
First we need to see that $C_\psi$ is well defined, namely that $\psi$ is a holomorphic function with $\psi(B_{c_0^2}) \subset B_{c_0^2}$. The holomorphy of $\psi$ follows from its definition. Let us define $F_1^{(j)}=\mathcal B^{-1}(D_1^{(j)})$ so that $\psi_1(z,w)=(F_1^{(j)}(z,w))$. Since
$\phi(\C_+^2)\subset \C_+^2$, 
$|D_1^{(j)}(s,t)|=p_j^{-\re \phi_1(s,t)}<1$ for every $(s,t)\in\C_+^2$, so that $\|F_1^{(j)}\|_\infty=\|D_1^{(j)}\|_\infty\leq 1$.
Assume by contradiction that there exists some $(z_0,w_0)\in B_{c_0^2}$ such that $\psi_1(z_0,w_0)$ does not belong to $c_0$. Then there exists an increasing sequence of integers
$(j_r)$ and $\e>0$ such that, for all $r\geq 1$, $|F_1^{(j_r)}(z_0,w_0)|\geq\e$.
By Montel's theorem, we may extract from $(F_1^{(j_r)})$ a sequence, that we will still denote $(F_1^{(j_r)})$, converging uniformly on compact subsets of $B_{c_0^2}$ to some
$F\in H^\infty(B_{c_0^2})$. Set $D=\mathcal BF$, so that $(D_1^{(j_r)})$ converges uniformly to $D$ on each product of half-planes $\CC_\e^2$. Now, for $(s,t)\in\C_+$,
\[ D_1^{(j_r)}(s,t)=p_{j_r}^{-\phi_1(s,t)}\xrightarrow{r\to+\infty}0 
\]
since $\re \phi_1(s,t)>0$. Thus $D$ hence $F$ are identically zero. But this contradicts $|F_1^{(j_r)}(z_0,w_0)|\geq\e$ for all $r\geq 1$. Finally, this yields that $\psi_1(B_{c_0^2})\subset c_0.$

To see that $T$ and $C_\psi$ are the same operator we will define a topology on $H^\infty(B_{c_0^2})$ so that the finite polynomials on $B_{c_0^2}$ are dense in $H^\infty(B_{c_0^2})$ and such that $T$ and $C_\psi$ are continuous, with the aim of extending \eqref{eq:T_and_C_psi_coincide_on_finite_polynomials_double} by continuity. Define $G:\C_+ \rightarrow B_{c_0}$ by $G(s) = \frac{1}{\mathfrak{p}^s}$ and consider $\tilde{\tau}$ the topology of uniform convergence on the product of half-planes $\overline{\C_{\sigma_1}} \times \overline{\C_{\sigma_2}}$ for $\HCdos$ and for $H^\infty(B_{c_0^2})$ we consider $\tau$ the topology of the uniform convergence on the compact subsets of $B_{c_0}$ of the form $K_{\sigma_1, \sigma_2} = \overline{\left\{ \left(\frac{1}{\mathfrak{p}^s},\frac{1}{\mathfrak{p}^t}\right) : \re s \geq \sigma_1, \re t \geq \sigma_2\right\}} = \overline{G(\overline{\C_{\sigma_1}})} \times  \overline{G(\overline{\C_{\sigma_2}})}= K_{\sigma_1} \times K_{\sigma_2}$. It should be noted that these topologies define metrizable spaces since we can take $\sigma_1 = \frac{1}{n}$, $\sigma_2= \frac{1}{m}$, $n,m \in \N$, and we get the same topologies. First, since for every $\sigma_1, \sigma_2 > 0$ there exists some $0 < r < 1$ such that $\sup_{ \substack{\re s \geq \sigma_1 \\ \re t \geq \sigma_2}} \left\| \left(\frac{1}{\mathfrak{p}^s},\frac{1}{\mathfrak{p}^t}\right) \right\|_\infty \leq r$, then any $f \in H^\infty(B_{c_0^2})$ has a uniformly convergent Taylor series on $K_{\sigma_1, \sigma_2}$. Moreover, by adapting the arguments from the proof of \cite[Theorem~2.5]{Aron17}, the set of finite polynomials on $B_{c_0^2}$ is dense in the space of homogeneous polynomials on $B_{c_0^2}$ with the topology induced by $\Vert\cdot\Vert_\infty$. Therefore we can extend \eqref{eq:T_and_C_psi_coincide_on_finite_polynomials_double} to homogeneous polynomials. Again, since the topology of $\Vert\cdot\Vert_\infty$ is finer than topology $\tau$, if we prove that $T$ and $C_\psi$ are continuous with the topology $\tau$, then we will be able to extend \eqref{eq:T_and_C_psi_coincide_on_finite_polynomials_double} to $H^\infty(B_{c_0^2})$ to get that $T=C_\psi$ as operators of $H^\infty(B_{c_0^2})$.

To see that $T$ is continuous, we just have to check that $C_\phi$ is continuous for the topology $\tilde{\tau}$ and that $\mathcal{B}$ defines a homeomorphism with the respective topologies $\tau$ and $\tilde{\tau}$. To prove that $C_\phi$ is continuous we have to apply Lemma~\ref{range_phi_comp_op_double} to $\phi$ to get that for every $\sigma_1, \sigma_2 > 0$ there exists some $\delta(\sigma_1), \delta(\sigma_2) > 0$ such that $\phi(\C_{\sigma_1} \times \C_{\sigma_2}) \subset \C_{\delta(\sigma_1)} \times \C_{\delta(\sigma_2)}$. Now, let $\{D_n\}_n \subset \HCdos$ be a sequence convergent to $D \in \HCdos$ with $\tilde{\tau}$. As 
$$\Vert C_\phi(D_n) - C_\phi(D) \Vert_{\C_{\sigma_1} \times \C_{\sigma_2}} = \Vert D_n \circ \phi - D \circ \phi \Vert_{\C_{\sigma_1} \times \C_{\sigma_2}} \leq \Vert D_n - D \Vert_{\C_{\delta(\sigma_1)} \times \C_{\delta(\sigma_2)}},$$
$C_\phi$ is continuous. Now, to see that $\mathcal{B}$ is a homeomorphism with the respective topologies, suppose that $\{f_n\}_n \subset H^\infty (B_{c_0^2})$ is a sequence convergent to $f \in H^\infty (B_{c_0^2})$ with $\tau$. Then, using the continuity of $f_n$ and $f$,
\begin{align*}
\Vert\mathcal{B}(f_n) - \mathcal{B}(f)\Vert_{\C_{\sigma_1} \times \C_{\sigma_2}}& = \sup_{ \substack{\re s \geq \sigma_1 \\ \re t \geq \sigma_2}} |\mathcal{B}(f_n)(s,t) - \mathcal{B}(f)(s,t) |\\
& = \sup_{\substack{\re s \geq \sigma_1 \\ \re t \geq \sigma_2}} \left|f_n\left(\frac{1}{\mathfrak{p}^s},\frac{1}{\mathfrak{p}^t}\right) - f\left(\frac{1}{\mathfrak{p}^s},\frac{1}{\mathfrak{p}^t}\right) \right|\\
&= \sup_{(z,\omega) \in K_{\sigma_1,\sigma_2}} |f_n(z,\omega) - f(z,\omega)| \\
& = \Vert f_n - f \Vert_{K_{\sigma_1,\sigma_2}},
\end{align*}
so clearly $\mathcal{B}$ is a homeomorphism between the topological spaces $(H^\infty(B_{c_0^2}), \tau)$ and $(\HCdos, \tilde{\tau})$, giving that $T$ is continuous with the topology $\tau$.

It remains to prove that $C_\psi$ is continuous with the topology $\tau$. Let us recall that $\psi=( (\mathcal{B}^{-1}(D^{(j)}_1))_j,(\mathcal{B}^{-1}(D^{(j)}_2))_j)$. Hence, if $\sigma_1, \sigma_2 > 0$, take $(z,\omega) \in G(\overline{\C_{\sigma_1}}) \times G(\overline{\C_{\sigma_2}})$, that is, $(z, \omega) = (\frac{1}{\mathfrak{p}^s},\frac{1}{\mathfrak{p}^t})$ for some $(s,t) \in \overline{\C_{\sigma_1}} \times \overline{\C_{\sigma_2}}$, and then $\mathcal{B}^{-1}(D^{(j)}_1)(z,\omega) = D_1^{(j)}(s,t) = \frac{1}{p_j^{\phi_1(s,t)}}$, so 
$$\mathcal{B}^{-1}(D^{(j)}_1)(G(\overline{\C_{\sigma_1}}) \times G(\overline{\C_{\sigma_2}})) 
\subset  \left\{ \frac{1}{p_j^{\phi_1(s,t)}} : \re s \geq \sigma_1, \re t \geq \sigma_2\right \} \subset \left\{ \frac{1}{p_j^s} : \re s \geq \delta(\sigma_1)\right \},$$
and therefore $\psi_1(K_{\sigma_1, \sigma_2}) \subset \overline{ \left\{ \frac{1}{\mathfrak{p}^s} : \re s \geq \delta(\sigma_1) \right\} } = K_{\delta(\sigma_1)}$. Analogously, $\psi_2(K_{\sigma_1, \sigma_2}) \subset K_{\delta(\sigma_2)}$, so $\psi(K_{\sigma_1, \sigma_2}) \subset K_{\delta(\sigma_1)} \times K_{\delta(\sigma_2)} = K_{\delta(\sigma_1),\delta(\sigma_2)}$. Then, if $\{f_n\}_n \subset H^\infty (B_{c_0^2})$ is a sequence convergent to $f \in H^\infty (B_{c_0^2})$ with $\tau$,
$$ \Vert C_\psi(f_n) - C_\psi(f) \Vert_{K_{\sigma_1,\sigma_2}} = \Vert f_n \circ \psi - f \circ \psi \Vert_{K_{\sigma_1,\sigma_2}} \leq \Vert f_n - f \Vert_{K_{\delta(\sigma_1),\delta(\sigma_2)}},$$
which gives the continuity of $C_\psi$ with the topology $\tau$.

Finally, as every function in $H^\infty(B_{c_0^2})$ is the limit of a series of finite polynomials with the topology $\tau$, and by \eqref{eq:T_and_C_psi_coincide_on_finite_polynomials_double}, $T$ and $C_\psi$ coincide on finite polynomials, $T = C_\psi$.

\end{proof}

\section{Superposition operators}

In this section we deal with superposition operators, which are defined by $S_\varphi(f) = \varphi \circ f$ where $f$ belongs to some function space and $\varphi$ is defined on $\C$.  We are especially interested in superposition operators acting on spaces of Dirichlet series, precisely on $\mathcal{H}^p$ and on $\HC$.
Let us first recall the characterization of superposition operators on the classical Hardy spaces $H^p(\D)$, which was given in \cite{Camera95}.

\begin{theorem} \label{characterization_superposition_operators_H^p(D)}
An entire function $\varphi$ defines a superposition operator $S_\varphi : H^p(\D) \rightarrow H^q(\D)$ if and only if $\varphi$ is a polynomial of degree at most $\lfloor \frac{p}{q} \rfloor$.
\end{theorem}

\begin{remark}
As a matter of fact, the assumption of $\varphi$ to be entire can be dropped.
If $\varphi:\C\to\C$ induces a superposition operator $S_\varphi$ mapping $H^p(\D)$ into $H^q(\D)$, then, since $f(z)=rz$ belongs to $H^p(\D)$ for all $r>0$,
the function $z\mapsto \varphi(rz)$ is analytic in $\D$ for all $r>0$, hence $\varphi$ is entire.
\end{remark}

Recall that the Bohr transform induces an isometric isomorphism from $H^p(\T^\infty)$ onto $\mathcal{H}^p$. The  subspace of $\mathcal{H}^p$ consisting of Dirichlet series of the form $\sum_{k=1}^\infty a_{2^k} \frac{1}{(2^k)^s}$ is isometrically isomoporphic to $H^{p} (\mathbb{T}) \hookrightarrow H^{p} (\mathbb{T}^{\infty})$. Let us recall also that $H^{p}(\mathbb{T})$ is isometrically isomorphic to $H^{p} (\mathbb{D})$. We are going to see that the superposition operators on  $\mathcal{H}^p$ are in fact the same ones as on $H^{p} (\mathbb{D})$.

\begin{theorem}
A function $\varphi:\C\to \C$ defines a superposition operator $S_\varphi : \mathcal{H}^p \rightarrow \mathcal{H}^q$ if and only if $\varphi$ is a polynomial of degree at most $\lfloor \frac{p}{q} \rfloor $.
\end{theorem}

\begin{proof}
We first see that if $\varphi$ is a polynomial of degree $N \leq \lfloor \frac{p}{q} \rfloor$ then $S_{\varphi}$ defines a superposition operator. Take first $\varphi_k(w) = w^k$. By Young's inequality we have $a^n \leq \frac{n}{p} a^p$ whenever $a>0$ and $\frac{n}{p}<1$. Then, since $\frac{kq}{p} \leq 1$ for all $1 \leq k \leq N$, given a Dirichlet polynomial $P$, we get 
\begin{align*}
\left\|\varphi_k(P)\right\|_q^q &= \lim_{T \rightarrow \infty} \frac{1}{2T} \int_{-T}^T |\varphi_k(P(it))|^q dt  \\ 
&=  \lim_{T \rightarrow \infty} \frac{1}{2T} \int_{-T}^T \left | (P(it))^k \right |^q dt  \\
& \leq \lim_{T \rightarrow \infty} \frac{1}{2T} \int_{-T}^T \frac{kq}{p} |P(it)|^{p} dt\\
&\leq \Vert P\Vert_p^{p},
\end{align*}
so $\varphi_k(P) \in \mathcal{H}^q$. Now, if $\varphi(w) = \sum_{k=0}^N b_k w^k$, then $\varphi(P) = \sum_{k=0}^N b_k \varphi_k(P) \in \mathcal{H}^q$ for every Dirichlet polynomial $P \in \mathcal{H}^p$. Using that Dirichlet polynomials are dense in $\mathcal{H}^p$ and $\mathcal{H}^q$ with the respective norms is enough to get the desired result.\\

On the other hand, suppose now that $\varphi$ defines a superposition operator $S_\varphi : \mathcal{H}^p \rightarrow \mathcal{H}^q$. First of all, since $s\mapsto \varphi\left(\frac{\sqrt{2} R}{2^{-s}}\right)$ is holomorphic in $\mathbb{C}_{1/2}$, then taking two different branches of the complex logarithm we get that $\varphi$ is holomorphic in $B(0,R) \setminus \{0\}$ for every $R>0$. 
Moreover, $s\mapsto \varphi\left(\frac{\sqrt{2} R}{2^{-s}}\right)$ is an absolutely convergent Dirichlet series, and therefore is bounded in $\C_\sigma$ for every $\sigma>1$.
Since $0$ is an isolated singularity of $\varphi$ we get that $\varphi$ is entire. Now let $f \in H^p(\D)$, $\mathcal{B}(f) \in \mathcal{H}^p$ and $\varphi \circ \mathcal{B}(f) \in \mathcal{H}^q$.
Since the Taylor series of $\varphi$ converges absolutely on $\C$, the Taylor series of $f$ converges absolutely on $\D$ and the Dirichlet series $\mathcal{B}(f)(s)$ converges absolutely in $\C_\sigma$ 
for any $\sigma >1$ then there exists $\{a_n\}_n$ such that $(\varphi \circ \mathcal{B}(f))(s) = \sum_{n=0}^\infty \frac{a_n}{(2^n)^s}$ for every $s \in \C_\sigma$, $\sigma >1$, and also 
$(\varphi \circ f)(z) = \sum_{n=0}^\infty a_n z^n$ for every $z \in \D$. Hence $\varphi \circ f = \mathcal{B}^{-1}(\varphi \circ \mathcal{B}(f)) \in H^q(\D)$,
thus $\varphi$ defines a superposition operator $S_\varphi : H^p(\D) \rightarrow H^q(\D)$ and consequently it is a polynomial of degree at most $\lfloor \frac{p}{q} \rfloor $. 
\end{proof}

\begin{remark}
It is interesting to note here the differences between the spaces $\mathcal{H}^p$ and $\HC$ regarding superposition operators. 
While only polynomials of a certain degree will define superposition operators $S_\varphi : \mathcal{H}^p \rightarrow \mathcal{H}^q$, the fact that $\HC$ is an algebra gives trivially
that any polynomial defines a superposition operator on $\HC$. This leads easily to the fact that any entire function defines a superposition operator.
Indeed, since entire functions are uniformly approximated by their Taylor series on all compact sets, in particular on the image of any function of $\mathcal H^\infty(\C_+)$, if $\varphi(z) = \sum_{n=1}^\infty a_n z^n$ and $f\in\HC$,  then $S_\varphi(f)$ is the uniform limit of $S_{\varphi_N}(f) \in \HC$ where $\varphi_N=\sum_{n=1}^N a_n z^n$. Using that $\HC$ is complete is enough to get that $S_\varphi: \HC \rightarrow \HC$ is a superposition operator.
\end{remark}

\end{document}